\documentclass[10pt]{amsart}

\usepackage{pinlabel}
\usepackage{amsmath,amsfonts,amssymb,latexsym,mathrsfs,stmaryrd}
\usepackage{color}
\usepackage{graphicx}
\usepackage{graphics}
\usepackage{enumerate}
\usepackage{amscd} 
\usepackage{amsxtra}
\usepackage{epic,eepic}
\usepackage{hyperref} 
\usepackage{enumitem}
\usepackage{MnSymbol}
\usepackage{dsfont}
\usepackage[all]{xy}
\usepackage{marginnote}
\usepackage{tikz-cd}

\newtheorem{theorem}{Theorem}[section]
\newtheorem{prop}[theorem]{Proposition}

\newtheorem{thm}[theorem]{Theorem}
\newtheorem{lemma}[theorem]{Lemma}

\theoremstyle{definition} 
\newtheorem{defn}[theorem]{Definition}

\newtheorem{example}[theorem]{Example}

\newtheorem{remark}[theorem]{Remark}

\newcommand{\C}{\mathbb{C}}
\newcommand{\dd}{\mathbf{d}}
\newcommand{\cO}{\mathcal{O}}

\newcommand{\bM}{\overline{M}}

\newcommand{\bb}{\mathbb}

\newcommand{\Z}{\bb{Z}}

\newcommand{\bbP}{\bb{P}}
\newcommand{\fX}{\mathfrak{X}}
\newcommand{\HH}{\mathfrak{H}}
\newcommand{\qu}{/\kern-.7ex/}
\newcommand{\lqu}{\backslash \kern-.7ex \backslash}
\newcommand{\on}{\operatorname}

\newcommand{\Hom}{\on{Hom}}

\newcommand{\ev}{\on{ev}}

\newcommand{\NE}{\on{NE}}
 
\newcommand{\bt}{\mathbf{t}}

\title[Mirror theorems for root stacks and pair]{Mirror theorems for root stacks and relative pairs}

\author{Honglu Fan}
\address{Department of mathematics\\ ETH Z\"urich\\ R\"amistrasse 101\\8092 Z\"urich\\Switzerland}
\email{honglu.fan@math.ethz.ch}

\author{Hsian-Hua Tseng}
\address{Department of Mathematics\\ Ohio State University\\ 100 Math Tower, 231 West 18th Ave.\\Columbus\\ OH 43210\\ USA}
\email{hhtseng@math.ohio-state.edu}

\author{Fenglong You}
\address{Department of Mathematical and Statistical Sciences\\ 632 CAB \\ University of Alberta\\Edmonton\\ AB\\ T6G 2G1\\ Canada}
\email{fenglong@ualberta.ca}



\begin{document}
\date{\today}

\begin{abstract} 
Given a smooth projective variety $X$ with a smooth nef divisor $D$ and a positive integer $r$, we construct an $I$-function, an explicit slice of Givental's Lagrangian cone, for Gromov--Witten theory of the root stack $X_{D,r}$. As an application, we also obtain an $I$-function for relative Gromov--Witten theory following the relation between relative and orbifold Gromov--Witten invariants.
\end{abstract}

\maketitle 

\setcounter{tocdepth}{1}
\tableofcontents

\section{Introduction}

\subsection{Overview}
Mirror symmetry has been an important way of computing Gromov--Witten invariants.
A mirror theorem is usually stated as a formula relating certain generating function (the $J$-function) of genus-zero Gromov--Witten invariants of a variety and the period integral (or the $I$-function) of its mirror. Mirror theorem for quintic threefolds was first proven by A. Givental \cite{Givental96a} and Lian--Liu--Yau \cite{LLY}. Since then, genus-zero mirror theorems have been proven for many types of varieties: toric complete intersections \cite{Givental96b}, \cite{Iritani08}, \cite{Iritani17}, partial flag varieties \cite{BCFKvS}, toric fibrations \cite{Brown}, etc. 

On the other hand, Gromov--Witten theory of Deligne--Mumford stacks was introduced in \cite{CR}, \cite{AGV02}, \cite{AGV}. Since Deligne--Mumford stacks are natural ingredients in mirror symmetry, there have been a lot of works on their mirror theorems as well. These include, for example, \cite{JK}, \cite{CCIT}, \cite{CCFK}, \cite{JTY}, and so on.

Our motivation is to study Gromov--Witten theory of Deligne--Mumford stacks. 
The geometric structure of a smooth Deligne--Mumford stack $\mathcal X$ in relation with the coarse moduli space $X$, which we assume to be a scheme, can be factored into three stages. 
Let
\[ 
\mathcal X\rightarrow X
\] 
be the structure map of an orbifold $\mathcal X$ to its coarse moduli $X$. Following the discussion of \cite{Tseng17}, it can be factored into
\[
\mathcal X\rightarrow \mathcal X_{rig}\rightarrow \mathcal X_{can}\rightarrow X.
\]
The first map is given by rigidification \cite{ACV} that allows one to "remove" the generic stablilizer of $\mathcal X$. 
Under the map $\mathcal X\rightarrow \mathcal X_{rig}$, $\mathcal X$ is a gerbe over $\mathcal X_{rig}$.
Following \cite{GM}, under some mild hypothesis, the second map $\mathcal X_{rig}\rightarrow \mathcal X_{can}$ is a composition of the root construction of \cite{AGV} and \cite{Cadman}. 
In the third map, $\mathcal X_{can}$ is the canonical stack associated to $X$ in the sense of \cite[Section 4.1]{FMN}. The stacky locus of $\mathcal X_{can}$ is of codimension at least $2$. 

To study Gromov--Witten theory of a Deligne--Mumford stack, one may break it down according to this decomposition. Gromov--Witten theory of gerbes has been studied by X. Tang and the second author in \cite{TT14}, \cite{TT16}. As to the second map, it is shown in \cite{GM} that the root construction is essentially the only way to introduce stack structures in codimension $1$. Gromov--Witten invariants of root stacks are partially studied by the second and the third author in \cite{TY16}. We would like to obtain more precise results in this paper.

\subsection{Orbifold Gromov--Witten theory of root stacks}
In this paper, we study genus-zero mirror symmetry for root stacks over smooth projective varieties. 

More precisely, let $X$ be a smooth projective variety and $D$ be a smooth nef divisor. In this paper, 
we construct the $I$-function for the $r$-th root stack $X_{D,r}$ in terms of the $J$-function of its coarse moduli space $X$. The idea of finding the $I$-function is to
construct root stacks as hypersurfaces in toric stack bundles. The hypersurface contruction is given in Section \ref{sec:hyp}. Combining with orbifold quantum Lefschetz theorem (\cite{Tseng} or \cite{CCIT09} plus a similar argument of \cite{KKP}), the $I$-function\footnote{Following the custom in mirror theorems, the term ``$I$-function'' refers to an explicitly constructed slice of Givental's Lagrangian cone.} of root stacks can be constructed as a hypergeometric modification of the $I$-function of the toric stack bundle, which is further written (according to \cite{JTY}) in terms of the $J$-function of $X$. We have the following theorem.

\begin{theorem}[=Theorem \ref{main-theorem}]
The $I$-function for the root stack $X_{D,r}$:
\begin{align}
I_{X_{D,r}}(Q,t,z)=&\sum_{\dd\in \on{NE}(X)}
J_{X, \dd}(t,z)Q^{\dd}\left(\frac{\prod_{0<a\leq D\cdot \dd}(D+az)}{\prod_{\langle a\rangle=\langle D_r\cdot \dd\rangle,0<a\leq D_r\cdot \dd}(D_r+az)}\right)\mathbf 1_{\langle -D_r\cdot \dd\rangle}
\end{align}
lies in the Givental's Lagrangian cone $\mathcal L_{X_{D,r}}$ for the root stack $X_{D,r}$, where $\langle a\rangle$ is the fractional part of the rational number $a$.
\end{theorem}
In the formula, $D_r\subset X_{D,r}$ is the divisor corresponding to the substack isomorphic to an $r$-th root gerbe of $D$. In computations, one could simply replace $D_r$ by $D/r$.

Using the $S$-extended $I$-function for toric stack bundles in \cite{JTY}, we also state the mirror theorem for root stacks in terms of $S$-extended $I$-function.
\begin{theorem}[=Theorem \ref{thm-orb-extended}]
The $S$-extended $I$-function
\begin{align*}
&I_{X_{D,r}}^{S}(Q,x,t,z)\\
=&\sum_{\dd\in \on{NE}(X)}\sum_{(k_1,\ldots,k_m)\in (\mathbb Z_{\geq 0})^m}J_{X,\dd}(t,z)Q^{\dd}\frac{\prod_{i=1}^m x_i^{k_i}}{z^{\sum_{i=1}^m k_i}\prod_{i=1}^m(k_i!)}\times\\
&\left(\prod_{0<a\leq D\cdot \dd}(D+az)\right)
\notag  \left(\frac{\prod_{\langle a\rangle=\langle D_r\cdot \dd-\frac{\sum_{i=1}^mk_ia_i}{r}\rangle,a\leq 0}(D_r+az)}{\prod_{\langle a\rangle=\langle D_r\cdot \dd-\frac{\sum_{i=1}^mk_ia_i}{r}\rangle,a\leq D_r\cdot \dd-\frac{\sum_{i=1}^mk_ia_i}{r}}(D_r+az)}\right)\mathbf 1_{\langle -D_r\cdot \dd+\frac{\sum_{i=1}^mk_ia_i}{r}\rangle}
\end{align*}
lies in the Givental's Lagrangian cone $\mathcal L_{X_{D,r}}$ for the root stack $X_{D,r}$.
\end{theorem}

We refer the readers to Section \ref{sec:hyper-cons} for the precise definitions of the notation.
Note that if $X$ is a toric variety, and $D$ is \emph{not} an invariant divisor, we can still write down the $I$-function of the corresponding root stack in terms of combinatorial data of $X$. Hence, our mirror formula provides a way to compute Gromov--Witten invariants of some non-toric stacks.

From a pure Gromov--Witten theory point of view, $I$-functions need to be accompanied with other reconstruction theorems in order to (sometimes partially) recover genus-zero Gromov--Witten theory. In Section \ref{subsec:refine}, a more general result describing the full genus-zero theory is obtained by further studying the localization on the hypersurface construction.

\begin{thm}[=Theorem \ref{thm:tw}]
Let $Y_{X_\infty,r}$ be the $r$-th root stack of $Y:=\bbP(\cO_X(-D)\oplus\cO_X)$ along the infinity section $X_\infty$. Given cohomological classes $\gamma_j\in H_{\on{CR}}^*(Y_{X_\infty,r})$ and nonnegative integers $a_j$, for $1\leq j\leq n$, we have
\[
\langle \prod_{j=1}^n \tau_{a_j}(i^*\gamma_j) \rangle^{X_{D,r}}_{0,n,\dd} = \left[
\displaystyle\int_{\bM_{0,n}(\fX_r,\dd)}
\prod\limits_{j=1}^n \bar{\psi}_j\on{ev}_j^* (i^{\prime*}\gamma_j) \frac{e_{\C^*}(\cO(D)_{0,n,\dd})}{e_{\C^*}(\cO(D/r)_{0,n,\dd})}
\right]_{\lambda=0},
\]
where $i:X_{D,r}\hookrightarrow Y_{X_\infty,r}$ and $i^\prime:\fX_r\hookrightarrow Y_{X_\infty,r}$ are the embeddings.
\end{thm}
Here on the right hand side, $\fX_r$ is the gerbe of $r$-th roots of $\cO(D)$ over $X$,
and $\cO(D)_{0,n,\dd}\in K^0(\bM_{0,n}(\fX_r,\dd))$ is given by pulling back $\cO(D)$ to the universal curve
and pushforward to the moduli space. $\cO(D/r)$ is the universal sheaf on the root gerbe, and similarly $\cO(D/r)_{0,n,\dd}$ is the pull-back and push-forward. $\cO(D)$ and $\cO(D/r)$ have fiberwise $\C^*$-action of weights $1$ and $1/r$, respectively. $\lambda$ is the corresponding equivariant parameter.

In short, this theorem says that genus-zero Gromov--Witten theory of a root stack is nothing but the non-equivariant limit of a (doubly) twisted theory on a root gerbe. This can be used to provide a second proof of previous mirror theorem of root stacks.

\subsection{Relation to relative Gromov--Witten theory}
Gromov--Witten theory of root stacks is important also because it is naturally related to relative Gromov--Witten theory. Abramovich-Caman-Wise \cite{ACW} proved that genus-zero Gromov--Witten invariants of root stacks are equal to the corresponding relative Gromov--Witten invariants if the roots are taken to be sufficiently large. The relation between higher genus invariants has recently been carried out by the second and the third authors in \cite{TY18a} and \cite{TY18b}. However, the correspondence between relative and orbifold invariants in \cite{ACW} and \cite{TY18b} only contains orbifold invariants whose orbifold markings are of small ages. That is, the ages are of the form $i/r$, for $r$ sufficiently large. In \cite{FWY}, the correspondence between genus-zero relative and orbifold invariants has been generalized to include orbifold invariants with large ages. The extra orbifold invariants correspond to relative invariants with negative contact orders as defined in \cite{FWY}. Givental's formalism for genus-zero relative Gromov--Witten theory has also been worked out in \cite{FWY}.

Although mirror symmetry for absolute Gromov--Witten theory has been intensively studied over the past two decades, mirror symmetry for relative Gromov--Witten theory has been missing in the literature for a long time. To the best of the authors' knowledge, mirror symmetry with relative Gromov--Witten invariants was first studied in M. van Garrel's thesis \cite{vG} for genus-zero relative Gromov--Witten theory of toric del Pezzo surfaces with maximal tangency along smooth effective anticanonical divisors.  

The relation between relative and orbifold invariants allows one to formulate relative mirror symmetry as a limit of mirror symmetry for root stacks. In Section \ref{rel-mirror}, we obtain a Givental style mirror theorem for relative Gromov--Witten invariants. That is, the $I$-function for relative Gromov--Witten invariants lies in Givental's Lagrangian cone for relative invariants, as defined in \cite{FWY}. In particular, the result in \cite{ACW} and \cite{TY18b} is already sufficient if one only considers the restricted $J$-function, defined in Section \ref{rel-mirror}, where each invariant only contains one relative marking. We obtain the $I$-function for relative Gromov--Witten invariants of $(X,D)$ with maximal tangency condition along the divisor $D$ by passing the $I$-function for root stacks to the limit. Therefore, we first obtain the relation between the restricted $J$-function of $(X,D)$ and the non-extended $I$-function of $(X,D)$ without using Givental's formalism for relative invariants.

\begin{theorem}[=Theorem \ref{thm-rel}]\label{intro-thm-rel}
Given a smooth projective variety $X$ and a smooth nef divisor $D$ such that the class $-K_X-D$ is nef. The non-extended $I$-function for relative Gromov--Witten invariants of $(X,D)$ is
\begin{align}
I_{(X,D)}(Q,t,z)
=\sum_{\dd\in \on{NE}(X)}J_{X, \dd}(t,z)Q^{\dd}\left(\prod_{0<a\leq D\cdot \dd-1}(D+az)\right)[\mathbf 1]_{-D\cdot \dd},
\end{align}
where we refer to Section \ref{rel-mirror} for the notation.
The non-extended $I$-function $I_{(X,D)}(Q,z)$ is equal to the restricted $J$-function $J_{(X,D)}([t^\prime]_0,z)$ for the relative Gromov--Witten invariants of $(X,D)$ after change of variables.
\end{theorem}

The non-extended $I$-function of $(X,D)$ coincides with the $I$-function for the local Gromov--Witten theory of $\cO_X(-D)$ (up to signs). Hence, Theorem \ref{intro-thm-rel} should be related to the relation between relative invariants of $(X,D)$ and local Gromov--Witten invariants of $\cO_X(-D)$ in \cite{vGGR} when the divisor is smooth and there is one relative marked point.

To state a mirror formula for relative invariants with more than one relative marked points, as well as relative invariants with negative contact orders, we consider the $S$-extended $I$-function $I_{(X,D)}^{S}(Q,x,t,z)$ for relative invariants. We refer the readers to Section \ref{sec:extended-rel-I} for the definition of $I_{(X,D)}^{S}(Q,x,t,z)$. The following theorem follows.
\begin{theorem}[=Theorem \ref{thm-rel-I-extended}]
The $S$-extended $I$-function $I_{(X,D)}^{S}(Q,x,t,z)$ for relative invariants lies in Givental's Lagrangian cone for relative invariants as defined in \cite[Section 7.5]{FWY}.
\end{theorem}

\subsection{Acknowledgment}
F.Y. would like to thank Qile Chen, Charles Doran and Melissa Liu for helpful discussions. H. F. is supported by grant ERC-2012-AdG-320368-MCSK and SwissMAP. H.-H. T. is supported in part by NSF grant DMS-1506551. F. Y. is supported by a postdoctoral fellowship funded by NSERC and Department of Mathematical Sciences at the University of Alberta.
\section{Preliminary}

\subsection{Orbifold Gromov--Witten theory}

In this section, we briefly recall the definition of genus-zero orbifold Gromov--Witten invariants and Givental's formalism. General theory of orbifold Gromov--Witten invariants can be found in \cite{Abramovich}, \cite{AGV02}, \cite{AGV}, \cite{CR} and, \cite{Tseng}.

Let $\mathcal X$ be a smooth proper Deligne--Mumford stack with projective coarse moduli space $X$. The Chen--Ruan orbifold cohomology $H^*_{\on{CR}}(\mathcal X)$ of $\mathcal X$ is the cohomology of the inertia stack $I\mathcal X$ with gradings shifted by ages. We consider the moduli stack $\overline{M}_{0,n}(\mathcal X, \dd)$ of $n$-pointed genus-zero degree $\dd$ stable maps to $\mathcal X$ with sections to gerbes at the markings (see \cite[Section 4.5]{AGV}, \cite[Section 2.4]{Tseng}). Given cohomological classes $\gamma_i\in H_{\on{CR}}^*(\mathcal X)$ and nonnegative integers $a_i$, for $1\leq i\leq n$, the genus-zero orbifold Gromov--Witten invariants of $\mathcal X$ are defined as follows
\begin{align}
\left\langle \prod_{i=1}^n \tau_{a_i}(\gamma_i)\right\rangle_{0,n,\dd}^{\mathcal X}:=\int_{[\overline{M}_{0,n}(\mathcal X, \dd)]^{w}}\prod_{i=1}^n(\on{ev}^*_i\gamma_i)\bar{\psi}_i^{a_i},
\end{align}
where, 
\begin{itemize}
\item $[\overline{M}_{0,n}(\mathcal X, \dd)]^{w}$ is the the weighted virtual fundamental class in \cite[Section 4.6]{AGV02} and \cite[Section 2.5.1]{Tseng}. 
    \item 
for $i=1,2,\ldots,n$,
\[
\on{ev}_i: \overline{M}_{0,n}(\mathcal X,\dd) \rightarrow I\mathcal X
\]
is the evaluation map;
\item
$\bar{\psi}_i\in H^2(\overline{M}_{0,n}(\mathcal X, \dd),\mathbb Q)$
is the descendant class.
\end{itemize}
The genus-zero Gromov--Witten potential of $\mathcal X$ is 
\[
\mathcal F_{\mathcal X}^0({\bf t}):= \sum_{n, \dd}\frac{Q^{\dd}}{n!}\langle {\bf t,\ldots, t}\rangle_{0,n, \dd}^{\mathcal X},
\]
where $Q$ is the Novikov variable and
\[
{\bf t}=\sum_{i\geq 0}t_iz^i\in H^*_{\on{CR}}(\mathcal X)[z].
\].

The Givental's formalism about the genus-zero orbifold Gromov--Witten invariants in terms of a Lagrangian cone in Givental's symplectic vector space was developed in \cite{Tseng}. The Givental's symplectic vector space is
\[
\mathcal H:=H^*_{\on{CR}}(\mathcal X,\mathbb C)\otimes \mathbb C[\![\on{NE}(\mathcal X)]\!][z,z^{-1}]\!],
\]
where $\on{NE}(\mathcal X)$ is the Mori cone of $\mathcal X$. The symplectic form on $\mathcal H$ is defined as
\[
\Omega(f,g):=\on{Res}_{z=0}(f(-z),g(z))_{\on{CR}}dz,
\]
where $(-,-)_{\on{CR}}$ is the orbifold Poincar\'e pairing of the Chen-Ruan cohomology $H^*_{\on{CR}}(\mathcal X)$ of $\mathcal X$.

We consider the polarization
\[
\mathcal H=\mathcal H_+\oplus \mathcal H_-,
\]
\[
\mathcal H_+=H^*_{\on{CR}}(\mathcal X,\mathbb C)\otimes \mathbb C[\![\on{NE}(\mathcal X)]\!][z], \quad \mathcal H_-=z^{-1}H^*_{\on{CR}}(\mathcal X,\mathbb C)\otimes \mathbb C[\![\on{NE}(\mathcal X)]\!][\![z^{-1}]\!].
\]
The Givental's Lagrangian cone $\mathcal L_{\mathcal X}$ is defined as the graph of the differential of $\mathcal F^0_{\mathcal X}$ in the dilaton-shifted coordinates. That is,
\[
\mathcal L_{\mathcal X}:=\{(p,q)\in \mathcal H_-\oplus \mathcal H_+| p=d_q\mathcal F^0_{\mathcal X}\} \subset \mathcal H.
\]
An important slice of $\mathcal L_{\mathcal X}$ is the $J$-function:
\[
J_{\mathcal X}(t,z):=z+t+\sum_{n, \dd}\sum_{\alpha}\frac{Q^{\dd}}{n!}\left\langle \frac{\phi_\alpha}{z-\bar{\psi}},t,\ldots,t\right\rangle_{0,n+1, \dd}\phi^{\alpha},
\]
where 
\[
\{\phi_\alpha\}, \{\phi^\alpha\}\subset H^*_{\on{CR}}(\mathcal X)
\]
are additive bases dual to each other under orbifold Poincar\'e pairing and, 
\[
t=\sum_{\alpha}t^\alpha\phi_\alpha\in H^*_{\on{CR}}(\mathcal X).
\]
One can decompose the $J$-function according to the degree of curves
\[
J_{\mathcal X}(t,z)=\sum_{\dd}J_{\mathcal X, \dd}(t,z)Q^{\dd}.
\]

Givental-style mirror theorems can be formulated in terms of the Lagrangian cone $\mathcal L_{\mathcal X}$. For example, the mirror theorem in \cite{Brown} (resp. \cite{JTY}) for toric fibrations (resp. toric stack bundles) states that the $I$-function, certain hypergeometric modification of the $J$-function of the base, lies in Givental's Lagrangian cone for the target. For toric complete intersection stack $\mathcal Y$ cut out by a generic section of a convex vector bundle $\mathcal E$ with suitable assumptions (see \cite[Section 5]{CCIT14}), the $I$-function lies in Givental's Lagrangian cone for $\mathcal Y$.

We will also need to consider Gromov--Witten invariants twisted by the $\mathbb T$-equivariant (inverse) Euler class. We assume that $\mathbb T$ acts trivially on the target $\mathcal X$. Let $E\rightarrow \mathcal X$ be a vector bundle equipped with a $\mathbb T$-linearization. We consider 
\[
E_{0,n,\dd}:=R\pi_*\on{ev}^*E\in K_{\mathbb T}^0(\overline{M}_{0,n}(\mathcal X, \dd)),
\]
where $\pi$ and $\on{ev}$ give the universal family:
\begin{displaymath}
    \xymatrix{ \mathcal C_{0,n,\dd} \ar[r]^{\on{ev}}\ar[d]^{\pi} & \mathcal X\\
  \overline{M}_{0,n}(\mathcal X, \dd)&}.
\end{displaymath}
The $(e_{\mathbb T}, E)$-twisted genus-zero Gromov--Witten invariants are defined to be
\begin{align}
\left\langle \prod_{i=1}^n \tau_{a_i}(\gamma_i)\right\rangle_{0,n,\dd,(e_{\mathbb T}, E)}^{\mathcal X}:=\int_{[\overline{M}_{0,n}(\mathcal X, \dd)]^{w}}\prod_{i=1}^n(\on{ev}^*_i\gamma_i)\bar{\psi}_i^{a_i}e_{\mathbb T}(E_{0,n,\dd}).
\end{align}
The $(e_{\mathbb T}^{-1}, E)$-twisted invariants are defined by replacing $e_{\mathbb T}(E_{0,n,\dd})$ with $e_{\mathbb T}^{-1}(E_{0,n,\dd})$.

\subsection{Root construction}
Given a smooth Deligne--Mumford stack $\mathcal X$, an effective smooth divisor $\mathcal D\subset \mathcal X$, and a positive integer $r$, the $r$-th root stack of $(\mathcal X,\mathcal D)$ is denoted by $\mathcal X_{\mathcal D,r}$. The Deligne--Mumford stack $\mathcal X_{\mathcal D,r}$ can be defined as the stack whose objects over a scheme $f: S\rightarrow \mathcal X$ consist of triples
\[
(M,\phi,\tau)
\]
where
\begin{enumerate}
\item $M$ is a line bundle over $S$,
\item $\phi:M^{\otimes r}\cong f^*\cO_{\mathcal X}(\mathcal D)$ is an isomorphism,
\item $\tau$ is a section of $M$ such that $\phi(\tau^r)$ is the tautological section of $f^*\cO_{\mathcal X}(\mathcal D)$.
\end{enumerate}

There is a natural morphism $f:\mathcal X_{\mathcal D,r}\rightarrow \mathcal X$, which is an isomorphism over $\mathcal X\setminus \mathcal D$.

When $\mathcal X$ is a scheme $X$, the inertia stack of the root stack $X_{D,r}$ can be decomposed into a disjoint union of $r$ components
\[
I(X_{D,r})=X_{D,r}\sqcup D_r\sqcup\cdots \sqcup D_r,
\]
where there are $r-1$ twisted components of $\mu_r$-gerbes $D_r$ over $D$. Roughly speaking, we see it as a way to add stacky structure to $X$ along $D$. 

The root stack is the root of a line bundle with a section. In Section \ref{subsec:refine}, we will also consider the root of a line bundle $L$ over a scheme $X$ which is called a root gerbe $\fX_r$ over $X$. Objects of $\fX_r$ over a scheme $f: S\rightarrow X$ consist of a line bundle $M$ over $S$ and an isomorphism 
\[
\phi:M^{\otimes r}\cong f^*L.
\]
The stack $\fX_r$ is a gerbe over $X$ banded by $\mu_r$. We refer to \cite[Appendix B]{AGV} for more detailed discussions about root constructions.

\section{A mirror theorem for root stacks}\label{sec:hyper-cons}

Let $X$ be a smooth projective variety and $D\subset X$ be a smooth nef divisor. We want to construct the root stack $X_{D,r}$ as a hypersurface in a weighted projective bundle, then apply known mirror theorems to obtain the $I$-function.
By \cite[Section E]{CCGK} and \cite{LLW}, the blow-up of $X$ along a smooth center can be constructed as a complete intersection in a projective bundle $Y$ over $X$. Motivated by their construction, we describe root stacks in a similar way. 





\subsection{A hypersurface construction}\label{sec:hyp}

Consider the line bundle 
\[
L:=\cO_X(-D),
\] 
and the projectivization 
\[Y:=\bbP(L\oplus\cO_X) \xrightarrow{\pi} X.
\]
Geometrically, $Y$ is the compactification of the total space of $L$. The zero section is 
\[
X_0=\bbP(\cO_X)\subset \bbP(L\oplus\cO_X).
\]
Also let 
\[
X_\infty= \bbP(L)\subset \bbP(L\oplus\cO_X).
\]
$X_\infty$ is the exact divisor we add in at infinity to compactify $L$. There is an invertible sheaf $\cO_Y(1)$ corresponding to the section $X_\infty$ (as a divisor class). Denote
\[
h=c_1(\cO_Y(1)).
\]
Choose a section $\sigma$ of $\cO_X(D)$ whose vanishing locus is the divisor $D$. This section defines a homomorphism of sheaves 
\[
f=(\sigma\oplus 1)^*\in \Hom_X(L\oplus\cO_X,\cO_X).
\]
The section $f$ determines a section 
\[
\tilde{f}\in H^0(Y,\cO_Y(1)),
\]
by the canonical identifications
\[
\Hom_X(L\oplus\cO_X,\cO_X)=H^0(X,(L\oplus\cO_X)^*)=H^0(X,\pi_*\cO_Y(1))=H^0(Y,\cO_Y(1)).
\]
Following the argument in \cite[Lemma E.1]{CCGK}, we can determine the zero locus of the section $\tilde{f}$ of $\cO_Y(1)$ locally over $X$. Locally, the zero locus of the section $\tilde{f}$ of $\cO_Y(1)$ is defined by $s_0\sigma(x)+s_1=0$, where $s_0$ is the local equation of $X_0$ and $s_1$ is the local equation of $X_\infty$. If $s_0=0$, then $s_1= 0$, this can not happen. If $s_0$ is not $0$, then $\sigma(x)=-s_1/s_0$. Hence, $\tilde{f}^{-1}(0)=\sigma(X)\cong X$. Its intersection with $X_\infty$ is $D$.

\begin{center}
\begin{tikzcd}
& \cO_Y(1) \arrow[d] \\
X\cong\sigma(X)=\tilde{f}^{-1}(0) \arrow[r, hook, "i"]
& Y:=\bbP(L\oplus\cO_X) \arrow[d,"\pi"] \arrow[u, bend left, "\tilde{f}"]\\
& X
\end{tikzcd}.
\end{center}
The above construction can be viewed as a special case of the geometric construction of blow-up when the blow-up center is the divisor $D$. The blown-up variety is simply $X$. 

To provide some intuition, a geometric explanation could go as follows. Note that the normal bundle of $X_\infty$ is $L^*$. Moreover, the total space of $L^*$ actually embeds into $Y$ as a Zariski neighborhood of $X_\infty$. Recall that $\sigma$ is a section of $\cO_X(D)\cong L^*$. Under the embedding of $L^*$ into $Y$, the image $\sigma(X)$ becomes a section of $Y$ over $X$. One can check $\sigma(X)$ is the zero locus of $\tilde f$, and the image of $\sigma(X)$ intersect $X_\infty$ along $D$ simply because $\sigma$ vanishes at $D$. 

To construct the root stack $X_{D,r}$, we consider the projection
\[
p: Y_{X_\infty,r}\rightarrow Y
\]
from the stack of $r$-th roots of $Y$ along the infinity section $X_\infty$. 
The section $\tilde{f}$ pulls back to a section $p^*(\tilde{f})$ of $p^*(O_Y(1))$. The zero locus $p^*(\tilde{f})^{-1}(0)$ is the inverse image of $\tilde f^{-1}(0)$ via $p$. This inverse image is isomorphic to $X_{D,r}$.

\begin{center}
\begin{tikzcd}
& p^*(\cO_Y(1)) \arrow[r] \arrow[d]&\cO_Y(1) \arrow[d] \\
p^*\tilde{f}^{-1}(0)=X_{D,r} \arrow[r, hook, "i"]
& Y_{X_\infty,r} \arrow[d] \arrow[r, "p"] \arrow[u, bend left, "p^*\tilde{f}"] & Y\arrow[d,"\pi"] \arrow[u, bend left, "\tilde{f}"]\\
& X & X
\end{tikzcd}.
\end{center}

\subsection{Construction of the $I$-function}
In this section, the Novikov variable $q^n$ corresponds to $n$ times of fiber classes of $Y$ over $X$, and $Q^{\dd}$ corresponds to the image of $\dd$ under the embedding $\NE(X_\infty)\subset\NE(Y)$. We also represent a curve class $\beta\in \NE(Y)$ by $(\dd,n)$ under this decomposition.

Recall that a line bundle $F$ over $\mathcal X$ is called convex if $H^1(\mathcal C,f^*F)=0$ for all genus-zero orbifold stable maps $f:\mathcal C\rightarrow \mathcal X$. 
\begin{lemma}\label{lem:convex}
$p^*\cO_Y(1)$ is a convex line bundle.
\end{lemma}
\begin{proof}
$p^*\cO_Y(1)$ is convex because it is a pullback of a convex line bundle $\cO_Y(1)$ on the coarse moduli space $Y$ of $Y_{X_\infty,r}$.
\end{proof}

Recall that
\[
h=c_1(\cO_Y(1)).
\]
We can now construct the $I$-function of $X_{D,r}$ in terms of the $J$-function of $X$ as follows. $Y_{X_\infty,r}$ is a toric stack bundle over $X$ with the fibers isomorphic to the weighted projective line $\bbP^1_{1,r}$. Therefore, by the mirror theorem for toric stack bundles \cite{JTY}, \cite{You}, we can write the $I$-function of $Y_{X_\infty,r}$ as the hypergeometric modification of the $J$-function of $X$:

\begin{align*}
I_{Y_{X_\infty,r}}(Q,q,t,z)=&e^{(h\log q)/z}\sum_{\dd\in \on{NE}(X)}\sum_{n\geq 0}J_{X,\dd}(t,z)Q^{\dd}q^n\\
& \left(\frac{(\prod_{\langle a\rangle=\langle n/r\rangle,a\leq 0}(h/r+az))(\prod_{a\leq 0 }(h-D+az))}{(\prod_{\langle a\rangle=\langle n/r\rangle,a\leq n/r}(h/r+az))(\prod_{a\leq n-D\cdot \dd }(h-D+az))}\right)\mathbf 1_{\langle -n/r\rangle},
\end{align*}
where $\langle a\rangle$ is the fractional part of the rational number $a$.

Recall that $i$ is the embedding $X_{D,r}\hookrightarrow Y_{X_\infty,r}$. In view of Lemma \ref{lem:convex}, we may apply orbifold quantum Lefschetz \cite{Tseng}, \cite{CCIT09}. Denote the resulting $I$-function for the root stack as the following.
\begin{align*}
&\tilde{I}_{X_{D,r}}(Q,q,t,z)\\
=&e^{(D\log q)/z}\sum_{\dd\in \on{NE}(X)}\sum_{n\geq 0}J_{X, \dd}(t,z)Q^{\dd}q^n\\
\notag & i^*\left(\frac{\prod_{\langle a\rangle=\langle n/r\rangle,a\leq 0}(h/r+az)}{\prod_{\langle a\rangle=\langle n/r\rangle,a\leq n/r}(h/r+az)}\frac{\prod_{a\leq 0 }(h-D+az)}{\prod_{a\leq n-D\cdot \dd }(h-D+az)}\frac{\prod_{a\leq n}(h+az)}{\prod_{a \leq 0}(h+az)}\right)\mathbf 1_{\langle -n/r\rangle}.
\end{align*}
Here ``$\sim$" on the top of $I_{X_{D,r}}$ is to signify that it is not completely intrinsic to $X_{D,r}$ due to the extra Novikov variable $q$.

We have the relation
\[
i^*h=D.
\]
Recall that $D_r$ is the substack of $X_{D,r}$ isomorphic to an $r$-th root gerbe of $D$. As divisor classes, $D_r=D/r$. The $I$-function $\tilde I_{X_{D,r}}$ can be written as
\begin{align}\label{equation-$I$-function}
\tilde I_{X_{D,r}}(Q,q,t,z)=&e^{(D\log q)/z}\sum_{\dd\in \on{NE}(X)}\sum_{n\geq D\cdot \dd}J_{X, \dd}(t,z)Q^{\dd}q^n\\
\notag & \left(\frac{\prod_{\langle a\rangle=\langle n/r\rangle,a\leq 0}(D_r+az)}{\prod_{\langle a\rangle=\langle n/r\rangle,a\leq n/r}(D_r+az)}\frac{\prod_{a\leq n}(D+az)}{\prod_{a \leq 0}(D+az)} \frac{1}{(n-D\cdot \dd)!z^{n-D\cdot \dd}}\right)\mathbf 1_{\langle -n/r\rangle}.
\end{align}
It is worth noting that $n$ starts from $D\cdot \dd$ instead of $0$. If $n<D\cdot \dd$, we have
\[
i^*\dfrac{\prod_{a\leq 0 }(h-D+az)}{\prod_{a\leq n-D\cdot \dd }(h-D+az)}=0\]
because of numerator has a zero factor when $a=0$. 

The above $I$-function contains an extra Novikov variable $q$ coming from the ambient $\mathbb{P}^1_{1,r}$-bundle. To get an intrinsic expression of $X_{D,r}$, we need to restrict Novikov variables to $\on{NE}(X_{D,r})\subset \on{NE}(Y_{X_\infty,r})$. 
To elaborate the situation, we start with the following standard computation.


\begin{lemma}\label{lem:curve_class}
If a curve class $\beta\in \on{NE}(Y_{X_\infty, r})$ represented by $(\dd, n)$ as above is contained in $\on{NE}(X_{D,r})$ then $n=D\cdot \dd$.
\end{lemma}
\begin{proof}
Suppose there is $\gamma\in \on{NE}(X_{D,r})$ such that $\beta=i_*\gamma$. Then $$n=\int_\beta c_1(\cO_Y(1))=\int_{i_*\gamma} h=\int_\gamma i^*h=(\pi^*D)\cdot \gamma=D\cdot (\pi_*\gamma)=D\cdot \dd.$$
\end{proof}
\if
The main concern in performing the restriction to $\on{NE}(X_{D,r})\subset \on{NE}(Y_{X_\infty,r})$ is Birkhoff factorization. Since the main reason for constructing the $I$-function is to use it to determine the $J$-function via Birkhoff factorization, we must make sure that the restriction process is compatible with Birkhoff factorization. 

The Birkhoff factorization procedure takes the form
$$U(z,z^{-1})=V(z^{-1})W(z),$$
where columns of $U$ are derivatives of the $I^{X_{D,r}}$, and columns of $V$ are derivatives of the twisted $J$-function of $Y_{X_\infty, r}$. This is an equality of formal series in Novikov variables. In other words, for each curve class $\beta \in \on{NE}(Y_{X_\infty, r})$, we have an equality $$U_\beta(z, z^{-1})=\sum_{\beta_1+\beta_2=\beta}V_{\beta_1}(z^{-1})W_{\beta_2}(z).$$ Here we write $U=\sum_\beta Q^\beta U_\beta$, $V=\sum_\beta Q^\beta V_\beta$, $W=\sum_\beta Q^\beta W_\beta$, and the splitting $\beta_1+\beta_2=\beta$ takes place in $\on{NE}(Y_{X_\infty,r})$.

\marginpar{HF: I like this argument. But one question. In the equation, $V,W$ might be composed with a mirror map. Why doesn't mirror map contain extra Novikov variables?}
{\it HHT: HF I think you're right, I overlooked the mirror map. I haven't changed the text here because I want to figure out the correct reasoning first. I think the idea is that we care about the cone of the twisted theory of $X_{D,r}$. Lemma \ref{lem:curve_class} says that $\on{NE}(X_{D,r})$ is contained in the affine space given by $n=D\cdot \dd$. So to recover the cone for the twisted theory, it is ok to first restrict to this affine space. The Birkhoff factorization procedure can change after restriction, but we do recover the cone of the twisted theory since the Novikov variables of $X_{D,r}$ are contained in the affine space. ThoughtS?}

\emph{ HF: Do you mean twisted theory of $Y_{X_\infty,r}$ instead of $X_{D,r}$? Not sure if I understand. But do you say that if we restrict the whole theory to the affine subspace containing $\on{NE}(X_{D,r})$, it satisfies those axioms of Lagrangian cone and thus forming a Lagrangian sub-cone? What makes me nervous is that in the topological recursion axiom, a curve in $\on{NE}(X_{D,r})$ might just break into two curves outside the affine space containing $\on{NE}(X_{D,r})$ (like $(d,0)+(0,D\cdot \dd)=(d,D\cdot \dd)$).}

\emph{Translating to our situation, I'm nervous about things like the following. Something like $Q^{\dd}$ might appear in mirror map. And then the Birkhoff factorization might involve $q^{D\cdot \dd}$. Both of these will be thrown away if we directly restrict the I-function to the affine space of $\on{NE}(X_{D,r})$. But they together influence the outcome. It seems hard to get rid of those by general argument (like degree reasons) because $D$ is quite arbitrary. I feel we probably need to do some induction. I agree that it feels like there should be a formal argument by playing with Lagrangian cone. But I can't think of a concrete idea so far.
}

The Euler class formula for virtual fundamental classes yields an explicit equality between the twisted $J$-function of $Y_{X_\infty,r}$ and the $J$-function of $X_{D,r}$, see for example, \cite[Section 5.2]{Tseng}. This formula implies that twisted $J$-function of $Y_{X_\infty,r}$ actually only depends on $\on{NE}(X_{D,r})$. If we restrict the Novikov variables in $U$ to $\on{NE}(X_{D,r})$, then in the splitting $\beta_1+\beta_2=\beta$ above, we have $\beta, \beta_1\in \on{NE}(X_{D,r})$. Lemma \ref{lem:curve_class} then implies that $\beta_2\in\on{NE}(X_{D,r})$. This shows that the restriction process is compatible with Birkhoff factorization.
\fi

Thus, by Lemma \ref{lem:curve_class}, a natural idea is to set $n=D\cdot \dd$ in \eqref{equation-$I$-function}. We replace $Q^{\dd}q^{D\cdot \dd}$ by $Q^{\dd}$ after rescaling $Q$. This yields the following expression.
\begin{align}\label{eqn:I}
I_{X_{D,r}}(Q,t,z)=&\sum_{\dd\in \on{NE}(X)}J_{X, \dd}(t,z)Q^{\dd}\times\\
&\left(\frac{\prod_{a\leq D\cdot \dd}(D+az)}{\prod_{a \leq 0}(D+az)}\right)
\notag  \left(\frac{\prod_{\langle a\rangle=\langle D_r\cdot \dd\rangle,a\leq 0}(D_r+az)}{\prod_{\langle a\rangle=\langle D_r\cdot \dd\rangle,a\leq D_r\cdot \dd}(D_r+az)}\right)\mathbf 1_{\langle -D_r\cdot \dd\rangle}.
\end{align}



Note that curve classes $(\dd,D\cdot \dd)$ lie on the boundary of the sub-cone $\{(\dd,n)\},n\geq D\cdot \dd$. We conclude that Birkhoff factorization and the mirror map of $\tilde I_{X_{D,r}}$ and the ones of $I_{X_{D,r}}$ have the same effect on the coefficient of $Q^{\dd}q^{D\cdot \dd}$ (corresponding to $Q^{\dd}$) after restricting the cohomology to the hypersurface $X_{D,r}\subset Y_{X_\infty,r}$. Recall that $D_r=D/r$ as a divisor class. 
Note that $D$ is a nef divisor (thus $D\cdot d\geq 0$).

\begin{theorem}\label{main-theorem}
Given a smooth projective variety $X$ and a smooth nef divisor $D$, the $I$-function of the $r$-th root stack $X_{D,r}$ of $(X,D)$ can be constructed as a hypergeometric modification of the $J$-function of $X$ as the following: 
\begin{align}
I_{X_{D,r}}(Q,t,z)=\sum_{\dd\in \on{NE}(X)} J_{X, \dd}(t,z)Q^{\dd}\left(\frac{\prod_{0<a\leq D\cdot \dd}(D+az)}{\prod_{\langle a\rangle=\langle D_r\cdot \dd\rangle,0<a\leq D_r\cdot \dd}(D_r+az)}\right)\mathbf 1_{\langle -D_r\cdot \dd\rangle},
\end{align}
where $\langle a\rangle$ is the fractional part of the rational number $a$.
The $I$-function\footnote{We need to require the cohomology classes at orbifold marking to be in $i^*H^*(X)$ instead of $H^*(D)$, because we apply quantum Lefschetz.} $I_{X_{D,r}}$ lies in Givental's Lagrangian cone $\mathcal L_{X_{D,r}}$ for $X_{D,r}$.
\end{theorem}

\begin{remark}
 There are two reasons that we need to assume $X$ is a smooth projective variety instead of a smooth proper Deligne-Mumford stack. First, the mirror theorem for toric stack bundles in \cite{JTY} requires the base to be a smooth variety. Whether it holds for smooth Deligne-Mumford stack is currently not yet studied. Second, quantum Lefschetz principle can fail for orbifold hypersurfaces if the line bundle is not the pullback from the coarse moduli space \cite{CGIJJM}. Provided a mirror theorem for toric stack bundles with stacky base, the above construction of $I$-function can be extended to the case when $X$ is a smooth Deligne--Mumford stack and $\mathcal O_X(D)$ is convex.
\end{remark}


\subsection{A refinement of the hypersurface construction}\label{subsec:refine}
In fact, the hypersurface construction in Section \ref{sec:hyp} can be exploited further.
We put a $\C^*$-action on fibers of $Y_{X_\infty,r}$ so that on
the normal bundle of $X_\infty$, the weight is $1$. More precisely, under the presentation $Y=\bbP(L\oplus \cO)$, we need the $\C^*$ to act on $L$ with weight $-1$ and to act trivially on $\cO$. It induces an action on $Y_{X_\infty,r}$. The $\C^*$-action naturally lifts to a $\C^*$-action on the line bundle $p^*\cO_Y(1)$.
Let $\lambda$ be the equivariant parameter. As mentioned in the previous section, orbifold quantum Lefschetz already implies a result stronger than the mirror theorem.
\begin{theorem}\label{thm:main}
Let $\langle\ldots\rangle^{Y_{X_\infty,r},p^*\cO_Y(1)}$ denote Gromov--Witten invariant twisted by equivariant line bundle $p^*\cO_Y(1)$. Given cohomological classes $\gamma_j\in H_{\on{CR}}^*(Y_{X_\infty,r})$ and nonnegative integers $a_j$, for $1\leq j\leq n$, we have
\[
\left\langle \prod_{j=1}^n \tau_{a_j}(i^*\gamma_j) \right\rangle^{X_{D,r}}_{0,n,\dd} = \left[\left\langle
 \prod_{j=1}^n \tau_{a_j}(\gamma_j) \right\rangle^{Y_{X_\infty,r},p^*\cO_Y(1)}_{0,n,i_*\dd}\right]_{\lambda=0},
\]
where $i:X_{D,r}\hookrightarrow Y_{X_\infty,r}$ is the embedding.
\end{theorem}

The theorem above already allows one to compute invariants of $X_{D,r}$ only using invariants of $X$. But we will do more in this section. Let us compute the right-hand side by localization. In the localization computation of invariants of $Y_{X_\infty,r}$, the fixed loci can be labelled by
decorated bipartite graphs. In particular, we have the graph $\Gamma_0$ consists of a
single vertex supported on $p^{-1}(X_\infty)$ without edges. This graph contributes
\[
\text{Cont}_{\Gamma_0}=\displaystyle\int_{[\bM_{0,n}(\fX_r,\dd)]^{\text{vir}}}
\prod\limits_{j=1}^n \on{ev}_j^*\gamma_j \frac{e_{\C^*}(\cO(D)_{0,n,\dd})}{e_{\C^*}(\cO(D/r)_{0,n,\dd})},
\]
where $\fX_r$ is the gerbe of $r$-th roots of $\cO(D)$ over $X$,
and $\cO(D)_{0,n,\dd}$ is given by pulling back $\cO(D)$ to the universal curve
and pushforward to the moduli space. $\cO(D/r)$ is the universal sheaf on the root gerbe, and $\cO(D/r)_{0,n,\dd}$ is the pullback and pushforward similar as before.

We claim that any other graph $\Gamma$ containing an edge always contributes $0$. The reason
comes from the effect of the twisting $e_{\C^*}(\left(p^*\cO_Y(1) \right)_{0,n,\dd})$ on an
edge. Schematically, the localization contribution of $\Gamma$ can be written as follows.
\[
\text{Cont}_{\Gamma_0}=\displaystyle\int_{[\bM_\Gamma]^{\text{vir}}} \prod\limits_{v\in V(\Gamma)}\text{Cont}_v \prod\limits_{e\in E(\Gamma)}\text{Cont}_e,
\]
where $\bM_{\Gamma}$ is the fixed locus in the moduli space, $V(\Gamma), E(\Gamma)$ are sets of vertices and edges, respectively. Writing a localization residue as vertex and edge contributions is standard, and a detailed expression in this situation can be written out by an easy modification of the computation in \cite{JTY}. But we do not need the full expression in our analysis.

Suppose an edge $e\in E(\Gamma)$ has multiplicity $k$. The corresponding edge contribution includes a factor coming from the twisting $e_{\C^*}(\left(p^*\cO_Y(1) \right)_{0,n,i_*\dd})$. More precisely, it has the form
\begin{equation}\label{eqn:edge}
\text{Cont}_e=\ev^*_e\prod\limits_{j=0}^k\left(\dfrac{k-j}{k}(D+\lambda)\right ) \text{Cont}'_e,
\end{equation}
where the factor $\prod\limits_{j=0}^k\left(\dfrac{k-j}{k}(D+\lambda)\right )$ comes from the twisting, and $\text{Cont}'_e$ refers to other factors (e.g., contribution of $T_{Y_{X_\infty,r}}$) whose expressions do not concern us. 
Note that when $j=k$,
$\text{Cont}_e=0$. We could almost conclude our claim. One could be a little more careful by checking the gluing of nodes at the $X_\infty$ side in order to make sure that this $0$ factor is not cancelled in the normalization sequence. It is straightforward and the conclusion is that $\text{Cont}_e=0$ whenever there is at least one edge.

As a result, we have the following statement.
\begin{thm}\label{thm:tw}
Given cohomological classes $\gamma_j\in H_{\on{CR}}^*(Y_{X_\infty,r})$ and nonnegative integers $a_j$, for $1\leq j\leq n$, we have
\[
\langle \prod_{j=1}^n \tau_{a_j}(i^*\gamma_j) \rangle^{X_{D,r}}_{0,n,\dd} = \left[
\displaystyle\int_{\bM_{0,n}(\fX_r,\dd)}
\prod\limits_{j=1}^n \bar{\psi}_j\on{ev}_j^* (i^{\prime*}\gamma_j) \frac{e_{\C^*}(\cO(D)_{0,n,\dd})}{e_{\C^*}(\cO(D/r)_{0,n,\dd})}
\right]_{\lambda=0},
\]
where $i:X_{D,r}\hookrightarrow Y_{X_\infty,r}$ and $i^\prime:\fX_r\hookrightarrow Y_{X_\infty,r}$ are the embeddings.
\end{thm}

In particular, this result covers Theorem \ref{main-theorem} as a special case, resulting in a second proof of Theorem \ref{main-theorem}. First, recall that $\underline{I}\fX_r=\sqcup_{j=0}^{r-1}X$. Let $\iota_j:X\rightarrow \underline{I}\fX_r$ be the isomorphism of $X$ with the component of age $j/r$. By \cite{AJT}, the small $J$-function of $X$ and the one of $\fX_r$ has a simple relation.
\begin{lemma}
$J_{\fX_r}(t,z)=r^{-1} \sum\limits_{j=0}^{r-1} (\iota_j)_*J_X(t,z).$
\end{lemma}
Write $J_{\fX_r}(t,z)= \sum\limits_{\dd} J_{\fX_r, \dd}(t,z)Q^{\dd}$. We can then construct $I$-function for the required twisted theory by attaching hypergeometric factors (\cite{Tseng,CCIT09}). By a slight abuse of notation, we use Novikov variable $Q^{\dd}$ for curve classes $\dd\in \NE(X_{D,r})$.
\begin{prop}
\begin{equation*}
\begin{split}    
I_{\fX_r}^{\text{tw}}(Q,t,z) = \sum\limits_{\dd}\sum\limits_\alpha J_{\fX_r,\dd}(t,z)Q^{\dd}&\left(\frac{\prod_{a\leq D\cdot \dd}(D+\lambda+az)}{\prod_{a \leq 0}(D+\lambda+az)}\right)\times\\
&\left(\frac{\prod_{\langle a\rangle=\langle (D\cdot \dd)/r\rangle,a\leq 0}(\frac{D+\lambda}{r}+az)}{\prod_{\langle a\rangle=\langle (D\cdot \dd)/r\rangle,a\leq (D\cdot \dd)/r}(\frac{D+\lambda}{r}+az)}\right), 
\end{split}
\end{equation*}
where superscript $\text{tw}$ means the (double-) twisted theory by $\cO(D)$ (using the characteristic class $e_{\C^*}(\cdot)$) and $\cO(D/r)$ (using $1/e_{\C^*}(\cdot)$).
\end{prop}
Combining all the above and taking nonequivariant limit, we obtain the following equation.



\begin{align}
I_{X_{D,r}}(Q,t,z)=\sum_{\dd\in \on{NE}(X)}&J_{X, \dd}(t,z)Q^{\dd}\left(\frac{\prod_{a\leq D\cdot \dd}(D+az)}{\prod_{a \leq 0}(D+az)}\right)\\
\notag & \left(\frac{\prod_{\langle a\rangle=\langle D_r\cdot \dd\rangle,a\leq 0}(D_r+az)}{\prod_{\langle a\rangle=\langle D_r\cdot \dd\rangle,a\leq D_r\cdot \dd}(D_r+az)}\right)\mathbf 1_{\langle -D_r\cdot \dd\rangle}.
\end{align}
In view of $D$ being a nef divisor, this becomes exactly the same equation as the one in Theorem \ref{main-theorem}.

\begin{remark}
Comparing Theorem \ref{thm:tw} with Theroem \ref{thm:main}, an advantage is that it can be made more explicit via Grothendieck-Riemann-Roch computation. In fact, the Hurwitz--Hodge classes $e(\cO(D/r)_{0,n, \dd})$ can be further related to double ramification cycles with target varieties. Further study is in progress.
\end{remark}
\begin{remark}
We would like to make a note about the choice of $\C^*$-weights on the bundles $\cO(D)$ and $\cO(D/r)$. In Theorem \ref{thm:main}, the $\C^*$-action on $p^*\cO_Y(1)$ is in fact arbitrary since it is nef and we only care about the non-equivariant limit. However, there is only one choice that allows us to further achieve Theorem \ref{thm:tw}. And the choice is the equivariant $\cO_Y(1)$ of $Y=\bbP(L\oplus\cO)$ with weights $-1,0$ on factors $L,\cO$. Otherwise, \eqref{eqn:edge} would look different and we would have more complicated graph sums. In the end, this forces the $\C^*$-weight on $\cO(D)$ to be $1$.
\end{remark}

\subsection{The $S$-extended $I$-function}

One can also consider the S-extended $I$-function for toric stack bundles to construct the S-extended $I$-function for root stacks. Then Theorem \ref{main-theorem} can also be stated in terms of the S-extended $I$-function for root stacks without change. We also refer to \cite{CCIT} and \cite{CCFK} for the S-extended I-function for toric stacks. As mentioned in \cite{Iritani}, \cite{CCIT}, \cite{JTY}, the non-extended $I$-function only determines the restriction of the $J$-function to the small parameter space $H^2(X_{D,r},\mathbb C)\subset H^{2}_{\on{CR}}(X_{D,r},\mathbb C)$. Taking the S-extended $I$-function allows one to determine the $J$-function along twisted sectors.

Recall that $Y_{X_\infty,r}$ is a toric stack bundle over $X$ and the fiber is the weighted projective line $\mathbb P^1_{1,r}$. As a toric Deligne--Mumford stack, the weighted projective line $\mathbb P^1_{1,r}$ can be constructed using stacky fans defined in \cite{BCS}. The fan sequence is
\[
0\longrightarrow\mathbb Z\xrightarrow{\left(\begin{matrix}r\\1\end{matrix}\right)}(\mathbb Z)^2\xrightarrow{(\begin{matrix}-1&r\end{matrix})}\mathbb Z.
\]

Following \cite{Jiang}, one can also consider the $S$-extended stacky fan. In general, we choose $S$ to be a subset of the so-called box elements of the toric stack bundle. For $\mathbb P^1_{1,r}$, it simply means we assume
\[
S:=\{a_1,a_2,\ldots,a_m\}\subset \{0,1,\ldots, r-1\}.
\]
The $S$-extended fan sequence is
\[
0\longrightarrow\mathbb Z^{1+m}\xrightarrow{\left(\begin{matrix}r& 0&\cdots &0\\1&-a_1&\cdots&-a_m\\0&r&\cdots&0\\
\vdots&\vdots&\cdots&\vdots\\0&0&\cdots&r\end{matrix}\right)}(\mathbb Z)^{2+m}\xrightarrow{(\begin{matrix}-1&r&a_1&\cdots&a_m\end{matrix})}\mathbb Z.
\]

By \cite{JTY}, the $S$-extended $I$-function is
\begin{align}\label{Extended-I-function}
&I_{X_{D,r}}^{S}(Q,x,t,z)
=\sum_{\dd\in \on{NE}(X)}\sum_{(k_1,\ldots,k_m)\in (\mathbb Z_{\geq 0})^m}J_{X, \dd}(t,z)Q^{\dd}\frac{\prod_{i=1}^m x_i^{k_i}}{z^{\sum_{i=1}^m k_i}\prod_{i=1}^m(k_i!)}\times\\
&\quad \left(\prod_{0<a\leq D\cdot \dd}(D+az)\right)
\notag  \left(\frac{\prod_{\langle a\rangle=\langle D_r\cdot \dd-\frac{\sum_{i=1}^mk_ia_i}{r}\rangle,a\leq 0}(D_r+az)}{\prod_{\langle a\rangle=\langle D_r\cdot \dd-\frac{\sum_{i=1}^mk_ia_i}{r}\rangle,a\leq D_r\cdot \dd-\frac{\sum_{i=1}^mk_ia_i}{r}}(D_r+az)}\right)\mathbf 1_{\langle -D_r\cdot \dd+\frac{\sum_{i=1}^mk_ia_i}{r}\rangle},
\end{align}
where $x=\{x_1,\ldots,x_m\}$ is the set of variables corresponding to the extended data $S$, and $D_r=D/r$ as divisor classes.

The following mirror theorem again follows from the mirror theorem for toric stack bundles \cite{JTY} and orbifold quantum Lefschetz \cite{Tseng}, \cite{CCIT09}.
\begin{theorem}\label{thm-orb-extended}
The $S$-extended $I$-function (\ref{Extended-I-function}) lies in the Givental's Lagrangian cone $\mathcal L_{X_{D,r}}$ of the root stack $X_{D,r}$. \end{theorem}
Theorem \ref{thm-orb-extended} allows one to compute orbifold Gromov--Witten invariants of root stacks when there are orbifold marked points with ages $a_i/r$, for $1\leq i\leq m$.

\subsection{Examples}
\begin{example}[Toric pairs]

We consider a toric pair $(X,D)$, where $X$ is a toric variety and $D$ is a toric divisor. The root stack $X_{D,r}$ is a toric Deligne--Mumford stack, where $r$ is a positive integer. The $I$-function for $X$ is
\[
I_X(Q,z)=ze^{\sum_{i=1}^l p_i\log Q_i/z}\sum_{\dd\in\on{NE}(X)}Q^{\dd}\prod_{i=1}^n\left(\frac{\prod_{a\leq 0}(D_i+az)}{\prod_{a \leq D_i\cdot \dd}(D_i+az)}\right),
\]
where $D_1,\ldots, D_n$ are toric divisors of $X$; $\{p_1,\ldots,p_l\}$ is a basis of $H^2(X,\mathbb Q)$. We can also write down the $I$-function for $X_{D,r}$:
\[
I_{X_{D,r}}(Q,z)= ze^{\sum_{i=1}^l p_i\log Q_i/z}\sum_{\dd\in\on{NE}(X)}Q^{\dd}\prod_{i=1}^n\left(\frac{\prod_{\langle a\rangle =\langle \mathcal D_i\cdot \dd\rangle, a\leq 0}(\mathcal D_i+az)}{\prod_{\langle a\rangle =\langle \mathcal D_i\cdot \dd\rangle,, a \leq \mathcal D_i\cdot \dd}(\mathcal D_i+az)}\right)\textbf{1}_{\langle-D_r\cdot \dd\rangle},
\]
where $\mathcal D_i$ are toric divisors of $X_{D,r}$.

Therefore,
\[
I_{X_{D,r}}(Q,z)
=\pi^*\sum_{\dd\in\on{NE}(X)} I_{X, \dd}(Q,z)\left(\frac{\prod_{a\leq D\cdot \dd}(D+az)}{\prod_{a \leq 0}(D+az)}\right)
\left(\frac{\prod_{\langle a\rangle =\langle D_r\cdot \dd\rangle, a\leq 0}(D_r+az)}{\prod_{\langle a\rangle =\langle D_r\cdot \dd\rangle,, a \leq D_r\cdot \dd}(D_r+az)}\right)\textbf{1}_{\langle-D_r\cdot \dd\rangle},
\]

where $\pi:X_{D,r}\rightarrow X$, and $D_r=D/r$ as divisor classes. This matches with the mirror formula in Theorem \ref{main-theorem}. Note that for toric pairs, $D$ is not required to be nef.
\end{example}

\begin{example}\label{example-root-plane-cubic}
Let $C$ be a smooth cubic curve in $\bbP^2$. The $I$-function for the root stack $\bbP^2_{C,r}$ is
\begin{align}
I_{\bbP^2_{C,r}}(Q,z)=ze^{H\log Q /z}\sum_{d\geq 0} Q^{d} \frac{\prod_{a=1}^{3d}(3H+az)}{\prod_{a=1}^d (H+az)^3}\frac{\prod_{\langle a \rangle=\langle 3d/r\rangle, a \leq 0}(3H/r+az)}{\prod_{\langle a \rangle=\langle 3d/r\rangle, a \leq 3d/r}(3H/r+az)}\textbf{1}_{\langle -3d/r\rangle},
\end{align}
where $H\in H^2(\bbP^2)$ is the hyperplane class.
\end{example}

\begin{example}\label{example-root-P3-cubic}
Let $S$ be a smooth cubic surface in $\bbP^3$. The $I$-function for the root stack $\bbP^3_{S,r}$ is
\begin{align}
I_{\bbP^3_{S,r}}(Q,z)=ze^{H\log Q /z}\sum_{d\geq 0} Q^{d} \frac{\prod_{a=1}^{3d}(3H+az)}{\prod_{a=1}^d (H+az)^4}\frac{\prod_{\langle a \rangle=\langle 3d/r\rangle, a \leq 0}(3H/r+az)}{\prod_{\langle a \rangle=\langle 3d/r\rangle, a \leq 3d/r}(3H/r+az)}\textbf{1}_{\langle -3d/r\rangle}.
\end{align}
More generally, one can also write down the $I$-function for the root stack $\bbP^n_{D,r}$, where $D$ is a smooth hypersurface of $\bbP^n$, in a similar way.
\end{example}

\section{A mirror theorem for relative pairs}\label{rel-mirror}

Genus-zero invariants of root stacks and genus-zero relative invariants are closely related. Following the formalism of \cite{FWY}, we rephrase our result as a mirror theorem of relative theory in this section.

Note that in \cite{FWY}, relative invariants with negative contact orders are defined, and it is proven to coincide with invariants of root stacks with some ``large"-age markings. In general, Theorem \ref{thm:main} should involve such invariants. In Section \ref{sec:small-rel-I}, we consider a non-extended $I$-function, and these invariants are not involved in this case. Hence, we state a mirror theorem for relative invariants without Givental formalism for relative invariants developed in \cite{FWY}. In Section \ref{sec:extended-rel-I}, we state the mirror theorem using Givental formalism for relative invariants and $S$-extended $I$-function to determine relative invariants with more than one relative markings, as well as invariants with negative markings.

\subsection{Genus-zero formalism of relative theory}
In this subsection, we briefly recall some notation in \cite{FWY}.

Define $\HH_0=H^*(X)$ and $\HH_i=H^*(D)$ if $i\in \Z - \{0\}$.
Let
\[
\HH=\bigoplus\limits_{i\in\Z}\HH_i.
\]
Each $\HH_i$ naturally embeds into $\HH$. For an element $\gamma\in \HH_i$, we denote its image in $\HH$ by $[\gamma]_i$. Define a pairing on $\HH$ by the following.
\begin{equation}\label{eqn:pairing}
\begin{split}
([\gamma]_i,[\delta]_j) = 
\begin{cases}
0, &\text{if } i+j\neq 0,\\
\int_X \gamma\cup\delta, &\text{if } i=j=0, \\
\int_D \gamma\cup\delta, &\text{if } i+j=0, i,j\neq 0.
\end{cases}
\end{split}
\end{equation}
The pairing on the rest of the classes is generated by linearity. 

If we pick fixed basis for $\HH_0$ and each $\HH_{i\neq 0}$, we have a basis for the whole $\HH$. $\HH$ is used as our ring of insertions, and the indices in $\HH_i$ signifies the contact order of the corresponding marking. For details, see \cite[Section 7.1]{FWY}.

In \cite{FWY}, relative invariants with insertions coming from $\HH$ are denoted by the following:
\[
I_\dd(\bar\psi^{a_1}[\gamma_1]_{i_1}, \ldots, \bar\psi^{a_n}[\gamma_n]_{i_n}),
\]
where $(i_1,\ldots,i_n)$ are contact orders and $\bar\psi^{a_1}\gamma_1,\ldots,\bar\psi^{a_n}\gamma_n$ are insertions (see \cite[Definition 7.3]{FWY}). Unfortunately, this notation is in conflict with the notation for $I$-functions in this paper. In this paper, we switch the above notation to the following:
\[
\left\langle \bar\psi^{a_1}[\gamma_1]_{i_1}, \ldots, \bar\psi^{a_n}[\gamma_n]_{i_n} \right\rangle_{0,n, \dd}^{(X,D)}.
\]

According to \cite[Section 7.5]{FWY}, the whole Lagrangian cone formalism can be built over $\HH$, and therefore, $I$-functions make sense as points on the Lagrangian cone. For $i\in \mathbb Z$, let $\{\widetilde T_{i,\alpha}\}$ be a basis for $\HH$ and $\widetilde T_{-i}^\alpha$ be the dual basis. For $l\geq 0$, we write $t_l=\sum\limits_{i,\alpha} t_{l;i,\alpha}\widetilde T_{i,\alpha}$, where $t_{l;i,\alpha}$ are formal variables. Also write
\[
\bt(z)=\sum\limits_{l=0}^\infty t_l z^l.
\]
\emph{The relative genus-zero descendant Gromov--Witten potential} is defined as
\[
\mathcal F(\bt(z))=\sum\limits_{\dd} \sum\limits_{n=0}^\infty \dfrac{Q^{\dd}}{n!} \langle\underbrace{\bt(\bar\psi),\ldots,\bt(\bar\psi)}_{n}\rangle_{0,n, \dd}^{(X,D)}.
\]
Givental's Lagrangian cone $\mathcal L$ is then defined as the graph of the differential $d\mathcal F$. More precisely, a (formal) point in Lagrangian cone can be explicitly written as
\[
-z+\bt(z)+\sum\limits_{\dd} \sum\limits_{n} \sum\limits_{i,\alpha} \dfrac{Q^\dd}{n!} \Big\langle\dfrac{\widetilde T_{i,\alpha}}{-z-\bar\psi},\underbrace{\bt(\bar\psi),\ldots,\bt(\bar\psi)}_{n}\Big\rangle_{0,n+1, \dd}^{(X,D)} \widetilde T_{-i}^\alpha.
\]
Moreover, \cite[Theorem 6.1]{FWY} implies
\begin{theorem}\label{thm:orbrel}
Let $\dd\in \NE(X)$ and $(i_1,\ldots,i_n)$ be a sequence of integers such that $\sum_k i_k=D\cdot \dd$. For $r\gg 1$, we have
\[
r^{\rho_-}\langle \tau_{a_1}(\gamma_1\mathbf 1_{i_1/r}),\ldots, \tau_{a_n}(\gamma_n\mathbf 1_{i_n/r}) \rangle_{0,n, \dd}^{X_{D,r}} = \langle \bar\psi^{a_1}[\gamma_1]_{i_1},\ldots \bar\psi^{a_n}[\gamma_n]_{i_n}\rangle_{0,n, \dd}^{(X,D)},
\]
where $\rho_-$ is the number of $i_k$, for $1\leq k \leq n$, such that $i_k<0$.
\end{theorem}


\subsection{The non-extended $I$-function for relative invariants}\label{sec:small-rel-I}

Assuming the anti-canonical divisor of $X_{D,r}$ is nef for $r\gg 1$, we consider the following $J$-function of $X_{D,r}$ with restricted parameters.
\[
J_{X_{D,r}}(t^\prime,z):=z+t^\prime+\sum_{\dd}\sum_{\alpha}Q^{\dd}\left\langle \frac{\phi_\alpha}{z-\bar{\psi}},t^\prime,\ldots,t^\prime\right\rangle^{X_{D,r}}_{0,n+1, \dd}\phi^\alpha,
\]
where we restrict the parameter $t^\prime$ to $H^*(X)\subset H^*_{\on{CR}}(X_{D,r})$; $\{\phi_\alpha\}$ is a basis of the ambient cohomology ring\footnote{the basis $\{\phi_\alpha\}$ is a basis of the cohomology ring pullback from the cohomological ring of $X_{D,r}$ to $D_r$.} of the twisted sector $\mathcal D$ of $I\mathcal X_{D,r}$ with age $(D\cdot \dd)/r$. Indeed, for this restricted $J$-function, the distinguished marked point (first marked point) has to be orbifold marked point with age $(D\cdot \dd)/r$ by virtual dimension constraint. 

On the other hand, under the set-up of \cite{FWY}, we can write down the restricted $J$-function of $(X,D)$ as follows.
\begin{defn}
The $J$-function for relative Gromov--Witten invariants of $(X,D)$ with restricted parameter is 
\[
J_{(X,D)}([t^\prime]_0,z):=z+[t^\prime]_0+\sum_{\dd}\sum_{\alpha}Q^{\dd}\left\langle \frac{[\phi_\alpha]_{D\cdot \dd}}{z-\bar\psi},[t^\prime]_0,\ldots,[t^\prime]_0\right\rangle^{(X,D)}_{0,n+1, \dd}[\phi^\alpha]_{-D\cdot \dd}.
\]
\end{defn}

Note that each invariant in $J_{(X,D)}([t^\prime]_0,z)$ only has one relative marking, hence the contact order is $D\cdot d$. Informally, $J_{(X,D)}([t^\prime]_0,z)$ can be seen as a ``limit" of $J^{X_{D,r}}(t^\prime,z)$ for $r\to \infty$. It is slightly different from the traditional limit, because our ``limit" also changes the underlying vector spaces from $H^*(I\mathcal X_{D,r})$ to $\HH$.

Consider $I_{X_{D,r}}(Q,t,z)$ from Theorem \ref{main-theorem}. We want to turn it into an $I$-function for relative Gromov--Witten theory. For a fixed $Q^{\dd}$, when $r>D\cdot \dd$, the coefficient looks like the following.
\begin{align}
&J_{X, \dd}(t,z)Q^{\dd}\left(\frac{\prod_{a\leq D\cdot \dd}(D+az)}{\prod_{a \leq 0}(D+az)}\right)\frac{1}{(D+(D\cdot \dd)z)/r}\mathbf 1_{\langle -(D\cdot \dd)/r\rangle}\\
\notag=& J_{X, \dd}(t,z)Q^{\dd}\left(\prod_{0<a\leq D\cdot \dd-1}(D+az)\right)r\mathbf 1_{\langle -(D\cdot \dd)/r\rangle},
\end{align}
Except $r\mathbf 1_{\langle -(D\cdot \dd)/r\rangle}$, the rest is independent of $r$ as long as $r>D\cdot \dd$. 
On nontrivial twisted sectors of $IX_{D,r}$, Poincar\'e pairing has an extra $r$ factor due to gerbe structures. Therefore, $r\mathbf 1_{\langle -(D\cdot \dd)/r\rangle}$ altogether turns into $[1]_{-D\cdot \dd}$.
Hence, we conclude
\begin{theorem} \label{thm-rel}
Given a smooth projective variety $X$ and a smooth nef divisor $D$ such that the class $-K_X-D$ is nef. The non-extended $I$-function for relative Gromov--Witten invariants of $(X,D)$ is
\begin{align}\label{$I$-function-rel}
I_{(X,D)}(Q,t,z)
= \sum_{\dd\in \on{NE}(X)}J_{X, \dd}(t,z)Q^{\dd}\left(\prod_{0<a\leq D\cdot \dd-1}(D+az)\right)[\mathbf 1]_{-D\cdot \dd},
\end{align}
where $[\mathbf 1]_{-D\cdot \dd}$ is the identity class of the component $\HH_{-D\cdot \dd}$ of the ring of insertions $\HH$.
The non-extended $I$-function $I_{(X,D)}(Q,t,z)$ equals the $J$-function $J_{(X,D)}([t^\prime]_0,z)$ for the relative Gromov--Witten invariants of $(X,D)$ with restricted parameter after change of variables.
\end{theorem}

The non-extended $I$-function for relative invariants coincides with the $I$-function for local invariants of $\cO_X(-D)$ (up to signs). It should be related to the equality in \cite{vGGR} between relative Gromov--Witten invariants of $(X,D)$ with one relative marking and the local Gromov--Witten invariants of $\cO_X(-D)$.



\begin{example}
Consider Gromov--Witten invariants of $\bbP^2$ relative to a generic cubic curve $C$ with maximal tangency along $C$ at a point. The $I$-function is the limit of the $I$- function in Example \ref{example-root-plane-cubic}:
\begin{align}
I_{(\bbP^2,C)}(Q,z)=ze^{H\log Q /z}\left(1+\sum_{d>0} Q^{d} \frac{\prod_{a=1}^{3d-1}(3H+az)}{\prod_{a=1}^d (H+az)^3}[\mathbf 1]_{-3d}\right).
\end{align}
The $I$-function $I_{(\bbP^2,C)}(Q,z)$ is equal to the $J$-function $J_{(\bbP^2,C)}([t^\prime]_0,z)$ via change of variables. Hence relative invariants of $(\bbP^2,C)$ with one marking can be computed. It coincides with the $I$-function for the canonical bundle $K_{\bbP^2}$ of $\bbP^2$ via the relation between relative invariants and local invariants \cite{vGGR}. Moreover, we recover the relative Gromov--Witten invariants of $(\bbP^2,C)$ with one relative marking computed in \cite{vG}.
\end{example}

\begin{example}
Similarly, we can consider Gromov--Witten invariants of $\bbP^3$ relative to a generic cubic surface $S$ with maximal tangency along $S$ at a point. The $I$-function can be obtained by taking the limit of the $I$-function in Example \ref{example-root-P3-cubic}: 
\begin{align}
I_{(\bbP^3,S)}(Q,z)=ze^{H\log Q /z}\left(1+\sum_{d>0} Q^{d} \frac{\prod_{a=1}^{3d-1}(3H+az)}{\prod_{a=1}^d (H+az)^4}[\mathbf 1]_{-3d}\right).
\end{align}
Relative invariants with one marking can be directly computed.
\end{example}

\subsection{The $S$-extended $I$-function for relative invariants}\label{sec:extended-rel-I}

Recall the $S$-extended $I$-function for root stacks is
\begin{align}
I_{X_{D,r}}^{S}(Q,x,t,z)
=\sum_{\dd\in \on{NE}(X)}\sum_{(k_1,\ldots,k_m)\in (\mathbb Z_{\geq 0})^m}J_{X, \dd}(t,z)Q^{\dd}\frac{\prod_{i=1}^m x_i^{k_i}}{z^{\sum_{i=1}^m k_i}\prod_{i=1}^m(k_i!)}\times \\
\notag \quad \left(\prod_{0<a\leq D\cdot \dd}(D+az)\right)\left(\frac{\prod_{\langle a\rangle=\langle D_r\cdot \dd-\frac{\sum_{i=1}^mk_ia_i}{r}\rangle,a\leq 0}(D_r+az)}{\prod_{\langle a\rangle=\langle D_r\cdot \dd-\frac{\sum_{i=1}^mk_ia_i}{r}\rangle,a\leq D_r\cdot \dd-\frac{\sum_{i=1}^mk_ia_i}{r}}(D_r+az)}\right)\mathbf 1_{\langle -D_r\cdot \dd+\frac{\sum_{i=1}^mk_ia_i}{r}\rangle}.
\end{align}
The $S$-extended $I$-function for root stacks determines the $J$-function along the twisted sectors of age $\frac{a_i}{r}$, for $1\leq i\leq m$. Therefore, for $r$ sufficiently large, the $S$-extended $I$-function determines relative invariants with relative markings of contact order $a_i$, for $1\leq i\leq m$.

The $S$-extended $I$-function for relative invariants can be obtained from $S$-extended $I$-function for root stacks by fixing $\{a_i\}_{i=1}^m$ and $d$ and letting $r$ be sufficiently large. The $S$-extended $I$-function splits into two parts:
\begin{itemize}
    \item When 
$\frac{\sum_{i=1}^mk_ia_i}{r}<D_r\cdot \dd$, we have
\[
\left(\frac{\prod_{\langle a\rangle=\langle D_r\cdot \dd-\frac{\sum_{i=1}^mk_ia_i}{r}\rangle,a\leq 0}(D_r+az)}{\prod_{\langle a\rangle=\langle D_r\cdot \dd-\frac{\sum_{i=1}^mk_ia_i}{r}\rangle,a\leq D_r\cdot \dd-\frac{\sum_{i=1}^mk_ia_i}{r}}(D_r+az)}\right)=\frac{1}{D_r+(D_r\cdot \dd-\frac{\sum_{i=1}^mk_ia_i}{r})z},
\]
where the extra factor $r$ on the right hand side together with $\mathbf 1_{\langle -D_r\cdot \dd+\frac{\sum_{i=1}^mk_ia_i}{r}\rangle}$ is identified with the class $[{\bf 1}]_{-D\cdot \dd+\sum_{i=1}^mk_ia_i}$ in relative theory.
\item When $\frac{\sum_{i=1}^mk_ia_i}{r}\geq D_r\cdot \dd$, we have 
\[
\left(\frac{\prod_{\langle a\rangle=\langle D_r\cdot \dd-\frac{\sum_{i=1}^mk_ia_i}{r}\rangle,a\leq 0}(D_r+az)}{\prod_{\langle a\rangle=\langle D_r\cdot \dd-\frac{\sum_{i=1}^mk_ia_i}{r}\rangle,a\leq D_r\cdot \dd-\frac{\sum_{i=1}^mk_ia_i}{r}}(D_r+az)}\right)=1.
\]
The class $\mathbf 1_{\langle -D_r\cdot \dd+\frac{\sum_{i=1}^mk_ia_i}{r}\rangle}$ is identified with the class $[{\mathbf 1}]_{-D\cdot \dd+\sum_{i=1}^mk_ia_i}$ in relative theory.
\end{itemize}
Therefore, we write the $S$-extended $I$-function for relative invariants as follows
\[
I_{(X,D)}^{S}(Q,x,t,z)=I_++I_-,
\]
where
\begin{align*}
I_+:=&\sum_{\substack{\dd\in \on{NE}(X),(k_1,\ldots,k_m)\in (\mathbb Z_{\geq 0})^m\\ \sum_{i=1}^mk_ia_i<D\cdot \dd} }J_{X, \dd}(t,z)Q^{\dd}\frac{\prod_{i=1}^m x_i^{k_i}}{z^{\sum_{i=1}^m k_i}\prod_{i=1}^m(k_i!)}\\
&\frac{\prod_{0<a\leq D\cdot \dd}(D+az)}{D+(D\cdot \dd-\sum_{i=1}^mk_ia_i)z}[{\mathbf 1}]_{-D\cdot \dd+\sum_{i=1}^mk_ia_i},
\end{align*}
and 
\begin{align*}
I_-:=&\sum_{\substack{\dd\in \on{NE}(X),(k_1,\ldots,k_m)\in (\mathbb Z_{\geq 0})^m\\ \sum_{i=1}^mk_ia_i\geq D\cdot \dd} }J_{X, \dd}(t,z)Q^{\dd}\frac{\prod_{i=1}^m x_i^{k_i}}{z^{\sum_{i=1}^m k_i}\prod_{i=1}^m(k_i!)}\\
&\left(\prod_{0<a\leq D\cdot \dd}(D+az)\right)[{\mathbf 1}]_{-D\cdot \dd+\sum_{i=1}^mk_ia_i}.
\end{align*}

The $S$-extended $I$-function for relative invariants correspond to $J$-function for relative invariants with possibly one negative relative marking. These relative invariants with one negative relative marking exactly correspond to the terms of $I_-$.
On the other hand, the terms of $I_+$
correspond to relative invariants without negative relative marking. Therefore, we have a mirror formula for relative invariants beyond the case of maximal tangency.

\begin{theorem}\label{thm-rel-I-extended}
The $S$-extended $I$-function $I_{(X,D)}^{S}(Q,x,t,z)$ for relative invariants lies in Givental's Lagrangian cone for relative invariants as defined in \cite[Section 7.5]{FWY}.
\end{theorem}

We can also allow some $a_i$ to be large, then the $S$-extended $I$-function allows us to compute relative invariants with more than one negative relative markings.

\end{document}